\documentclass[reqno]{amsart} 
\usepackage{amsfonts, amsmath, amsthm, amssymb, latexsym, graphicx}

\def\b{\mathbb }

\theoremstyle{plain}
\newtheorem{theorem}{Theorem}[section]
\newtheorem{corollary}[theorem]{Corollary}
\newtheorem{lemma}[theorem]{Lemma}
\newtheorem{proposition}[theorem]{Proposition}
\theoremstyle{definition}

\newtheorem{remark}[theorem]{Remark}

\theoremstyle{remark}

\numberwithin{equation}{section}

\newcommand{\mcl}[1]{\mathcal{#1}}

\newcommand{\bz}{\mathbf{z}}

\title[Covariance matrices of CMS models]{Central limit theorems
 for multivariate Bessel processes in the freezing regime II:\\
 the covariance matrices}

\author{Sergio Andraus}
\address{Faculty of Science and Engineering, Chuo University, Kasuga 1-13-27, Bunkyo-Ku, Tokyo 112-8551, Japan}
\email{andraus@phys.chuo-u.ac.jp}
\author{Michael Voit}
\address{Fakult\"at Mathematik, Technische Universit\"at Dortmund,
          Vogelpothsweg 87,
          D-44221 Dortmund, Germany}
\email{michael.voit@math.tu-dortmund.de}


\begin{document}
\date{\today}

\begin{abstract}
Bessel processes  $(X_{t,k})_{t\ge0}$ in $N$ dimensions
are classified via associated root systems and multiplicity constants $k\ge0$.
They describe 
 interacting Calogero-Moser-Suther\-land particle systems with $N$ particles  and are related to
 $\beta$-Hermite and $\beta$-Laguerre ensembles.
Recently, several central limit theorems  were derived for fixed 
$t>0$, fixed starting points, and $k\to\infty$. In this paper we extend the CLT in the A-case from  start in
0 to arbitrary starting distributions by using a limit result for the corresponding Bessel functions. 
We also determine the eigenvalues and eigenvectors of the covariance matrices of the Gaussian limits
and study applications to CLTs for the  intermediate particles
for $k\to\infty$ and then $N\to\infty$.
\end{abstract}

\maketitle

\section{Introduction} 

Integrable interacting particle systems of 
Calogero-Moser-Suther\-land type on $\mathbb R$ with $N$ particles 
 are described by multivariate 
Bessel processes on closed
Weyl chambers in $\mathbb R^N$. These processes are classified via root systems and a finite number of 
 multiplicity parameters $k$
which govern the interactions; see  \cite{CGY}, \cite{R1},  
 \cite{RV1}, \cite{RV2}, \cite{DF}, \cite{DV} and references therein.
Recently, several limit theorems were derived for these processes when one or several 
 multiplicity parameters $k$ tend to infinity;
see  \cite{AKM1}, \cite{AKM2}, \cite{AV}, \cite{V}, and \cite{VW}. In particular,  \cite{V} contains 
central limit theorems for  the  root systems $A_{N-1}$, $B_N$, and $D_N$ when the particles start
 in the origin $0\in\mathbb R$ 
 or, in some cases, with an arbitrary starting
 distribution independent from $k$. In \cite{V}, the  CLTs for $k\to\infty$ were derived for  the $A$-case only when
 the processes start in 0, while in all other cases  arbitrary starting
 distributions were possible. This shortcoming in \cite{V} in 
the $A$-case was caused by the lack of a suitable limit result for the Bessel functions of type $A$
for $k\to\infty$.
 We shall derive the corresponding limit result for the Bessel functions below 
 which then will lead to a CLT for  arbitrary starting
 distributions in Section 2. 

In all CLTs in \cite{V} and in Section 2 below, the limits in the CLTs are essentially independent  from the starting distributions,
 and usually, the limits are $N$-dimensional centered Gaussian distributions where the inverses $S_N:=\Sigma_N^{-1}$
of the covariance matrices
$\Sigma_N$ can be determined explicitely in terms of the zeros of certain classical orthogonal polynomials. For instance, in the case  $A_{N-1}$,
 the zeroes of the Hermite polynomial $H_N$ appear, and in the  case  $B_{N}$, the zeros of appropriate Laguerre polynomials appear.
We  determine the eigenvalues and eigenvectors of the matrices  $S_N$ and thus of $\Sigma_N$ in these cases.
The results are surprisingly simple.
These diagonalizations of  $S_N$ and  $\Sigma_N$ may be applied to
 the limit behavior of  the middle particle in the cases  $A_{N-1}$ when we first take
 $k\to\infty$ and then  $N\to\infty$. 
We  carry out the diagonalizations 
for the $A$-cases by using the empirical distributions $\mu_N$
 of the zeroes of the Hermite polynomials $H_N$ in combination with the finite systems of orthogonal polynomials
associated with the measure $\mu_N$ introduced in Section 3.
 Corresponding results for the $B$-case are presented in Section 4
for the multiplicities $(k_1,k_2)=(\nu\cdot\beta,\beta)$ with $\nu>0$ fixed and $\beta\to\infty$.
Here, the zeroes of classical Laguerre polynomials (with a parameter depending on $\nu$) 
instead of Hermite polynomials appear. In both cases, i.e., the Hermite as well as the Laguerre case, we get finite systems
of orthogonal polynomials depending on $N$ which converge for $N\to\infty$
 to the Tchebychev-polynomials of second kind which are orthogonal with respect to Wigner's semicircle distribution.

The results of this paper on Bessel processes with start in $0\in\mathbb R^N$ are closely related with 
 central limit theorems for $\beta$-Hermite and $\beta$-Laguerre ensembles for the spectra of tridiagonal random matrix models
due to Dumitriu and Edelman \cite{DE1}. In particular, our freezing results correspond in some cases  to the limits $\beta\to\infty$
in \cite{DE2}. In particular, \cite{DE2} contains explicit formulas for the covariance matrices $\Sigma_N$
 of the Gaussian limits while we here use explicit formulas for their inverses $S_N$ as in \cite{V}.
 In general, most of our results below for the starting point $x=0$ admit interpretations in random matrix theory; 
for the background here we refer to \cite{D}, \cite{Me}, as well as to \cite{RRV} for some specific results.

We also mention that
the Bessel processes are diffusions on Weyl chambers which satisfy some stochastic 
differential equations; see \cite{CGY} and references there. These SDEs are used in \cite{AV} and \cite{VW}
 to derive  strong laws of large numbers and functional central limit theorems for $X_{t,k}$ for  $k\to\infty$ with  strong
rates of convergence, whenever the processes start in points of the form $\sqrt k\cdot x$ where $x$
 is some point in the interior of the Weyl chamber. These limit theorems are even locally uniform in $t$.
It should be noticed that while the CLTs  in \cite{AV} and \cite{VW} may have different forms, 
in some cases similar Gaussian limits appear with covariance matrices which are closely related to the matrices $S_N$ and $\Sigma_N$.
Hence, the diagonalization results below admit applications for the CLTs in \cite{VW}.

\section{A central limit theorem in the $A$-case for arbitrary starting distributions}\label{CLT-A}

Consider the root system $A_{N-1}$ first. The associated Bessel processes $(X_{t,k})_{t\ge0}$ 
live on the closed Weyl chamber
$$C_N^A:=\{x\in \mathbb R^N: \quad x_1\ge x_2\ge\ldots\ge x_N\},$$
 the generator of the transition semigroup is 
\begin{equation}\label{def-L-A} L_Af:= \frac{1}{2} \Delta f +
 k \sum_{i=1}^N\Bigl( \sum_{j\ne i} \frac{1}{x_i-x_j}\Bigr) \frac{\partial}{\partial x_i}f ,
 \end{equation}
where we regard the multiplicity $k\in[0,\infty[$ as a parameter and we assume reflecting boundaries in the usual sense (see, for example, \cite[p. 97]{KS}).

We are interested in limit theorems for  $(X_{t,k})_{t\ge0}$
for fixed $t>0$ and  $k\to\infty$.
For this we recall that by \cite{R1},  \cite{RV1}, \cite{RV2}, the 
 transition probabilities are given for $t>0$,  $x\in C_N^A$, $S\subset C_N^A$ a Borel set, by
\begin{equation}\label{density-A}
K_t(x,S)=c_k^A \int_S \frac{1}{t^{\gamma_A+N/2}} e^{-(\|x\|^2+\|y\|^2)/(2t)} J_k^A\Big(\frac{x}{\sqrt{t}}, \frac{y}{\sqrt{t}}\Big) 
\cdot w_k^A(y)\> dy
\end{equation}
with
\begin{equation}\label{def-wk-A-gamma}
w_k^A(x):= \prod_{i<j}(x_i-x_j)^{2k}, \quad\quad \gamma_A=kN(N-1)/2, \end{equation}
and the Macdonald-Mehta-Opdam   constant
\begin{equation}\label{const-A}
 c_k^A:= \Bigl(\int_{C_N^A}  e^{-\|y\|^2/2}\cdot \prod_{i<j} (y_i-y_j)^{2k} \> dy\Bigr)^{-1}
=\frac{N!}{(2\pi)^{N/2}} \cdot\prod_{j=1}^{N}\frac{\Gamma(1+k)}{\Gamma(1+jk)}.
\end{equation}
Here,
$J_k^A$ is a multivariate Bessel function of type $A$ with multiplicity $k$;
 see e.g. \cite{R1},  \cite{AKM1}.
 We here only  recapitulate that 
$J_k^A$ is analytic on $\mathbb C^N \times \mathbb C^N $ with
$ J_k^A(x,y)>0$ for $x,y\in \mathbb R^N $, and with
 $J_k^A(x,y)=J_k^A(y,x)$ and $J_k^A(0,y)=1$
for  $x,y\in \mathbb C^N $. Further properties will be discussed below.

If we start in $x=0\in \mathbb R^N$, then $X_{t,k}$ has the  density
\begin{equation}\label{density-A-0}
 \frac{c_k}{t^{\gamma+N/2}} e^{-\|y\|^2/(2t)} \cdot w_k(y)\> dy
\end{equation}
on $C_N^A$ for $t>0$, 
which is in particular well-known for $k=1/2, 1,2$ and $t=1$ as the distribution of the ordered eigenvalues of
Gaussian orthogonal, unitary, and symplectic ensembles; see e.g. \cite{D}.
For general $k>0$ it is known from the tridiagonal $\beta$-Hermite ensembles of \cite{DE1}.

It is well-known (see \cite{AKM1} and also Section 6.7 of \cite{S}) that the density (\ref{density-A-0})
is maximal on $C_N^A$  precisely
 for $y=\sqrt 2\cdot {\bf z}$ where $ {\bf z}\in C_N^A$ is the vector
with the ordered zeros of the classical  Hermite polynomial $H_N$ as entries
where, as usual,  the polynomials $(H_N)_{N\ge 0}$ are orthogonal w.r.t.
 the density  $e^{-x^2}$. More precisely, we have  the following useful
characterization of the vector $ {\bf z}$; see \cite{AV}:

\begin{lemma}\label{char-zero-A}
 For $z\in C_N^A$, the following statements  are equivalent:
\begin{enumerate}
\item[\rm{(1)}] The function $W_A(x):=\sum_{i,j: i<j} \ln(x_i-x_j) -\|x\|^2/2$ is maximal at $\bz\in C_N^A$;
\item[\rm{(2)}] For $i=1,\ldots,N$:  $z_i= \sum_{j: j\ne i} \frac{1}{z_i-z_j}$;
\item[\rm{(3)}] $\bz=(z_{1,N},\ldots,z_{N,N})$ for
 the ordered zeros $z_{1,N}>\ldots> z_{N,N}$ of $H_N$.
\end{enumerate}
\end{lemma}

This characterization was used in \cite{V} to prove  the following central limit theorem (please notice that the limit
$N(0,t\cdot\Sigma_N)$ there must be replaced by $N(0,\Sigma_N)$): 

\begin{theorem}\label{clt-main-a}
Consider the  Bessel processes $(X_{t,k})_{t\ge0}$ of type $A_{N-1}$ on $C_N^A$ for $k\ge0$ with start in $0\in C_N^A$.
Then, for each $t>0$,
$$\frac{X_{t,k}}{\sqrt t} -  \sqrt{2k}\cdot (z_{1,N},\ldots,z_{N,N})$$
converges for $k\to\infty$ to the centered $N$-dimensional distribution $N(0,\Sigma_N)$
with the regular covariance matrix $\Sigma_N$ with $\Sigma_N^{-1}=S_N=(s_{i,j})_{i,j=1}^N$ and
\begin{equation}\label{covariance-matrix-A}
s_{i,j}:=\left\{ \begin{array}{r@{\quad\quad}l}  1+\sum_{l\ne i} (z_{i,N}-z_{l,N})^{-2} & \text{for}\quad i=j \\
   -(z_{i,N}-z_{j,N})^{-2} & \text{for}\quad i\ne j  \end{array}  \right.  . 
\end{equation}
 The matrix $S_N$ satisfies $det\> S_N =N!$.
\end{theorem}

We now extend this CLT to arbitrary starting points  $x\in C_N^A$ and even 
 arbitrary starting distributions on $ C_N^A$. However, the statement of the CLT will be slightly more complicated than for the other root systems in \cite{V},
 as the systems $A_{N-1}$ are not reduced on $\mathbb R^N$. This means that with the vector ${\bf 1}:=(1,\ldots,1)\in \mathbb R^N$,  
the space $\mathbb R^N$ can be decomposed into   $\mathbb R\cdot {\bf 1}$
and its orthogonal complement
 $$ {\bf 1}^\perp=\{x\in\mathbb R^N:\> \sum_i x_i=0\} \subset \mathbb R^N$$ 
where the associated Weyl group 
(which is the symmetric group $S_N$ here) acts on both spaces separately. 
It will turn out that the limit behavior of the CLT is slightly different on both components.
To describe this, 
we denote the orthogonal projections from  $ \mathbb R^N$ onto  $\mathbb R\cdot {\bf 1}$ and  $ {\bf 1}^\perp$ by
 $\pi_{\bf 1}$ and $\pi_{{\bf 1}^\perp}$ respectively. In particular, for all $x\in \mathbb R^N$,
 $\pi_{\bf 1}(x)= \bar x{\bf 1} $ for the center of gravity
 $\bar x=\frac{1}{N}\sum_{i=1}^N x_i$ of the particles.

\begin{theorem}\label{clt-main-a-general-x}
Consider the  Bessel processes $(X_{t,k})_{t\ge0}$ of type $A_{N-1}$ on $C_N^A$ for $k\ge0$ with a fixed starting point $x\in C_N^A$.
Then, for each $t>0$,
$$\frac{X_{t,k}}{\sqrt t} -  \sqrt{2k}\cdot (z_{1,N},\ldots,z_{N,N})$$
converges for $k\to\infty$ to the  $N$-dimensional normal distribution $N(\pi_{\bf 1}(x/\sqrt{t}),\Sigma_N)$
with $\Sigma_N$ as in Theorem \ref{clt-main-a}.
\end{theorem}

For the proof of Theorem \ref{clt-main-a-general-x}
 we mainly 
follow the ideas of the proofs of Theorem 3.3 and Corollary 3.7 in \cite{V} in the $B$-case.
 As main ingredient we need some facts on  $J_k^A$. We first recapitulate the following well-known 
decomposition; see e.g. \cite{BF}:

\begin{lemma}\label{decomposition-a}
For all $x,y\in \mathbb R^N$,
\begin{equation}\label{bessel-reduction-A}
 J_k^A( x,y)= e^{\langle\pi_{\bf 1}(x), \pi_{\bf 1}(y)\rangle}\cdot  J_k^A( \pi_{{\bf 1}^\perp}(x),
\pi_{{\bf 1}^\perp}(y))=e^{N\bar x\bar y}\cdot  J_k^A( \pi_{{\bf 1}^\perp}(x),
\pi_{{\bf 1}^\perp}(y)).
\end{equation}
\end{lemma}

We also need the following limit result for 
 $J_k^A$ for $k\to\infty$ which is a consequence of Corollary 8 of \cite{AM}
on Dunkl kernels for arbitrary root systems. Here, we include a proof
 that is specific to the root system  $A_{N-1}$:

\begin{theorem}\label{limit-Bessel-a}
For $x,y\in {\bf 1}^\perp$, 
\begin{equation}
\lim_{k\to\infty} J_k^A(\sqrt{2k}\cdot x,y)=\exp\Big(\frac{\|x\|^2\|y\|^2}{N(N-1)}\Big)\end{equation}
locally uniformly.
\end{theorem}

\begin{proof}
From \cite{BF} we have
\begin{equation}
J_k^A(x,y)={}_0\mcl{F}_0^{(1/k)}(x,y)=
\sum_{n=0}^\infty\sum_{\tau:l(\tau)\leq N, \>|\tau|=n}
\frac{c_\tau(1/k)}{c_\tau^\prime(1/k)}\frac{\mcl{P}_\tau^{(1/k)}(x)\mcl{P}_\tau^{(1/k)}(y)}{(k N)_\tau^{(1/k)}},
\end{equation}
with $\mcl{P}_\tau^{(\alpha)}(x)$ a Jack polynomial \cite{Ma} and $\tau$ an integer
 partition with dual partition $\tau^\prime$. In general, integer partitions are
 sequences of non-negative integers in non-strictly decreasing order, namely $\tau=(\tau_1,\tau_2,\ldots)$
 with $\tau_i\geq\tau_j$ for every $i<j$. Moreover, the length of the partition, denoted $l(\tau)$, is the
 number of nonzero parts in the partition, and the sum of its parts is denoted $|\tau|$. The dual partition
 $\tau^\prime$ is the partition with parts $\tau_i^\prime$ equal to the number of parts of $\tau$ that are greater
 than or equal to $i$. Finally, the expression $(i,j)\in \tau$ means that both $i\leq l(\tau)$ and $j\leq \tau_i$ 
are satisfied. With this, we can give the definition of all remaining symbols,
\begin{align}
(a)_\tau^{(\alpha)}&=\prod_{i=1}^{l(\tau)}\frac{\Gamma(a-(i-1)/\alpha+\tau_i)}{\Gamma(a-(i-1)/\alpha)},\notag\\
c_\tau(\alpha)&=\prod_{(i,j)\in\tau}(\alpha(\tau_i-j)+\tau_j^\prime-i+1),\notag\\
c_\tau^\prime(\alpha)&=\prod_{(i,j)\in\tau}(\alpha(\tau_i-j+1)+\tau_j^\prime-i).\notag
\end{align}
We rewrite  the generalized Pochhammer symbol as
\begin{equation}
(a)_\tau^{(\alpha)}=\prod_{i=1}^{l(\tau)}\frac{\Gamma(a-(i-1)/\alpha+\tau_i)}{\Gamma(a-(i-1)/\alpha)}=\prod_{(i,j)\in\tau}(a-(i-1)/\alpha+j-1).
\end{equation}
Now we consider the large $k$ limit for the coefficients of the sum,
\begin{align}
\frac{c_\tau(1/k)}{c_\tau^\prime(1/k)(k N)_\tau^{(1/k)}}
&=\prod_{(i,j)\in\tau}\frac{\tau_i-j+k(\tau_j^\prime-i+1)}{(\tau_i-j+1+k(\tau_j^\prime-i))(k(N-i+1)+j-1)}\notag\\
&=\prod_{(i,j)\in\tau:\> i<\tau^\prime_j}\frac{\tau_i-j+k(\tau_j^\prime-i+1)}{(\tau_i-j+1+
k(\tau_j^\prime-i))(k (N-i+1)+j-1)}\notag\\
&\quad \times\prod_{(i,j)\in\tau:\>i=\tau^\prime_j}\frac{\tau_i-j+k}{(\tau_i-j+1)(k (N-i+1)+j-1)}\notag\\
&=\prod_{(i,j)\in\tau:\>i<\tau^\prime_j}\Big(
\frac{1}{k}\frac{\tau_j^\prime-i+1}{(\tau_j^\prime-i)(N-i+1)}+O(k^{-2})\Big)\notag\\
&\quad\times\prod_{(i,j)\in\tau:\> i=\tau^\prime_j}\Big(\frac{1}{(\tau_i-j+1)(N-i+1)}+O(k^{-1})\Big).
\end{align}
Now, recall that the Jack polynomials are homogeneous,
\[\mcl{P}_\tau^{(1/k)}(\sqrt{2k}x)=(2k)^{|\tau|/2}\mcl{P}_\tau^{(1/k)}(x),\]
and that they converge to the elementary symmetric polynomials,
\begin{equation*}
e_n(x)=\sum_{1\leq l_1<\cdots<l_n}\prod_{j=1}^{n}x_{l_j},\ e_\tau(x)=\prod_{j=1}^{l(\tau)}e_{\tau_j}(x),
\end{equation*}
when $k\to\infty$,
\[\lim_{k\to\infty}\mcl{P}_\tau^{(1/k)}(x)=e_{\tau^\prime}(x).\]
Then, we have
\begin{align}
J_k^A(\sqrt{2k}\cdot x,y)&=
\sum_{n=0}^\infty \sum_{\tau: \> l(\tau)\leq N,\> |\tau|=n} 
\Big(\frac{1}{k}\Big)^{|\tau|-\tau_1}\notag\\
&\quad\times\prod_{(i,j)\in\tau:\> i<\tau^\prime_j}\Big(\frac{\tau_j^\prime-i+1}{(\tau_j^\prime-i)(N-i+1)}+O(k^{-1})\Big)\notag\\
&\quad\times\prod_{(i,j)\in\tau:\> i=\tau^\prime_j}\Big(\frac{1}{(\tau_i-j+1)(N-i+1)}+O(k^{-1})\Big)\notag\\
&\quad\times (2k)^{|\tau|/2}e_{\tau^\prime}(x)e_{\tau^\prime}(y).
\end{align}
However, since we have imposed $\sum_{i=1}^Nx_i=e_1(x)=0$, all terms for which any of the $\tau^\prime_j=1$ vanish automatically.
 Consequently, we must have $\tau_1=\tau_2$ for all partitions, and the leading-order terms in $k$ are those with partitions
x $\tau$ of length two with $\tau_1=\tau_2$. Therefore,
\begin{eqnarray}
\lim_{k\to\infty}J_k^A(\sqrt{2k}x,y)&=&\sum_{n=0}^\infty\frac{2^{2n}}{n!N^n(N-1)^n}[e_{2}(x)e_{2}(y)]^n\nonumber\\
&=&\exp\Big[\frac{4 e_{2}(x)e_{2}(y)}{N(N-1)}\Big].
\end{eqnarray}
Now, since
\[0=\Big(\sum_{i=1}^Nx_i\Big)^2=\sum_{i=1}^Nx_i^2+2\sum_{1\leq i<j\leq N}x_ix_j,\]
we have
\[\|x\|^2=-2e_2(x)\]
and a similar relation for $y$. Finally, we obtain
\begin{equation}
\lim_{k\to\infty}J_k^A(\sqrt{2k}x,y)=\exp\Big[\frac{\|x\|^2 \|y\|^2}{N(N-1)}\Big],
\end{equation}
as desired.
\end{proof}

\begin{proof}[Proof of Theorem \ref{clt-main-a-general-x}]
By the definition of the transition kernels $K_t$  in (\ref{density-A}),
 the $K_t$ admit the same
space-time-scaling as Brownian motions. We thus may assume that $t=1$ in the proof without loss of generality.

Moreover, (\ref{bessel-reduction-A}) implies 
 that the kernels 
 $K_t$ are partially translation invariant in the sense that
\begin{equation}\label{translation-invariance-A}
K_t(x+c{\bf 1}, S+c{\bf 1})= K_t(x,S) \quad\quad 
\text{for}\>\> c\in\mathbb R, \> ,t>0, \> x\in C_N^A,\> S\subset C_N^A.
\end{equation}
Thus, without loss of generality, we can add the assumption that the starting point $x\in C_N^A$
satisfies $x\in {\bf 1}^\perp$.

Then,
$X_{1,k}$ has the  density
$$c_k^A e^{-\|x\|^2/2-\|y\|^2/2}\cdot J_k^A(x,y) \cdot w_k^A(y)$$
on $C_N^A$. Hence, $X_{1,k}-\sqrt{2k}\cdot{\bf z}$  has the density
\begin{align}\label{a-density-detail}
f_k^A(y):=c_k^A & e^{-\|x\|^2/2} J_k^A(x,y+\sqrt{2k}\cdot{\bf z} )\\
&\cdot \exp\Bigl(-\|y +\sqrt{2k}\cdot{\bf z}\|^2/2\Bigr) w_k^A(y+\sqrt{2k}\cdot{\bf z})\notag
\end{align}
on the shifted cone $C_N^A-\sqrt{2k}\cdot{\bf z} $ with $f_k^A(y)=0$ elsewhere on $\mathbb R^N$.
 Using the definition of $w_k^A$ we now write this density as
$$f_k^A(y)= \tilde c_k \cdot h_k(y),$$
with $\tilde c_k$ given by (see Appendix~D in \cite{AKM1})
$$\tilde c_k:=e^{-\|x\|^2/2}\Big(\frac{k}{e}\Big)^{kN(N-1)/2}\frac{N!}{(2\pi)^{N/2}} \cdot\prod_{j=1}^{N}\frac{\Gamma(1+k)}{\Gamma(1+jk)} \prod_{m=1}^Nm^{km},$$
which is independent of $y$ (but dependent on  $x, k$), and with
\begin{align}\label{a-density-detail-defh}h_k(y):=&  J_k^A(x,y+\sqrt{2k}\cdot{\bf z} )\cdot\notag\\
  &\cdot \exp\Bigl(-\|y\|^2/2 -\sqrt{2k} \langle y,{\bf z}\rangle +
2k \sum_{i<j}\ln\bigl(1+ \frac{y_i-y_j}{\sqrt{2k}(z_i-z_j)}\bigr)\Bigr)\notag\\
=&  J_k^A(x,y+\sqrt{2k}\cdot{\bf z} )\cdot \exp\Bigl(-\|y\|^2/2-\frac{1}{2}\sum_{i<j}\frac{(y_i-y_j)^2}{(z_i-z_j)^2} + O(k^{-1/2})\Bigr)
\end{align}
for $y\in C_N^A-\sqrt{2k}\cdot{\bf z} $ and $h_k(y)=0$ elsewhere.
The last equality in (\ref{a-density-detail-defh}) follows from the Taylor formula for $\ln(1+x)$
and from Lemma \ref{char-zero-A} precisely
 as in the proof of Eq.~(2.8) of \cite{V}.
Next, we recall that $x\in {\bf 1}^\perp$ (by our assumption) and ${\bf z}\in {\bf 1}^\perp$ 
(because $H_N$ has either even or odd symmetry).
We thus conclude from Lemma \ref{decomposition-a} and
Theorem \ref{limit-Bessel-a} that for all $y\in \mathbb R^N$
\begin{align}\label{limit-besselfunction-a2}
\lim_{k\to\infty} J_{k}^A( x,y+\sqrt{2k}\cdot {\bf z}  )&=
\lim_{k\to\infty} J_{k}^A( x,\sqrt{2k}( {\bf z}+ y/\sqrt{2k}) )\notag\\
&=
 \exp\Bigl( \frac{\|x\|^2 \| {\bf z}\|^2}{N(N-1)}\Bigr)= \exp( \|x\|^2/2)
=: d(x)
\end{align}
where we have used 
\begin{equation}\label{hermite-zero-square-a}
\sum_{k=1}^N z_{k,N}^2= N(N-1)/2
\end{equation}
(see (D.22) in \cite{AKM1}).
In summary, 
\begin{equation}\label{summary-a1}
\lim_{k\to\infty}h_k(y)=d(x)\cdot  \exp\Bigl(-\frac{\|y\|^2}{2}-\frac{1}{2}\sum_{i<j}\frac{(y_i-y_j)^2}{(z_i-z_j)^2}
\Bigr).
\end{equation}

Now let $f\in C_b(\mathbb R^N)$ be a bounded continuous function. We shall show that 
 (\ref{summary-a1}) implies that
\begin{equation}\label{density-a-limit}
\lim_{k\to\infty}\int_{\mathbb R^N} f(y)\cdot h_k(y)\> dy = d(x)
\int_{\mathbb R^N} f(y)\cdot 
 \exp\Bigl(-\frac{\|y\|^2}{2}-\frac{1}{2}\sum_{i<j}\frac{(y_i-y_j)^2}{(z_i-z_j)^2}\Bigr)\> dy.
\end{equation}
For this we use dominated convergence. We consider
 the Taylor polynomial  of $\ln(1+x)$ and notice that by the Lagrange remainder,
\begin{equation}\label{ln-expansion-remainder-a}
\ln\bigl(1+ \frac{y_i\pm y_j}{\sqrt{\beta}(z_i\pm z_j)}\bigr) = \frac{y_i\pm y_j}{\sqrt{\beta}(z_i\pm z_j)}
 -\frac{(y_i\pm y_j)^2}{2\beta(z_i\pm z_j)^2}\cdot w_\pm
\end{equation}
with $w_\pm\in[0,1]$. As in the proof of Theorem 2.2 of \cite{V} we obtain 
from Lemma \ref{decomposition-a} that for all $k>0$
\begin{equation}\label{part-est-a}
0\le h_k(y)\le  J_{k}^A( x,\sqrt{2k}( {\bf z}+ y/\sqrt{2k}) )\cdot e^{-\|y\|^2/2}.
\end{equation}
Next, we estimate  $J_k^A$. For this
we recapitulate from \cite{RV2} that for all root systems and all multiplicities $k\ge0$, the associated Bessel functions $J$ 
satisfy 
$$ 0<J(a,b)\le \exp(\|a\|\cdot\|b\|) \quad\quad\text{for all}\quad\quad a,b\in\mathbb R^N.$$
In particular, 
$$ 0<J_k^A( x,y+\sqrt{2k}\cdot{\bf z})\le 
 \exp(\|x\|\cdot (\|y\|+\sqrt{2k}\cdot\|{\bf z} \|)).$$
This shows that
\begin{equation}\label{est-j-a}
J_k^A(x,y+\sqrt{2k}\cdot {\bf z})\le e^{2\|x\|\cdot \|y\|} \quad\text{for}\>\> k>0,\> \text{and }y\in \mathbb R^N
\>\> \text{with }\>\>  \|y\| \ge \sqrt{2k}\cdot  \| {\bf z} \| .
\end{equation}
On the other hand, if  $\|y\| \le \sqrt{2k}\cdot  \| {\bf z} \|$, then 
 $ y/\sqrt{2k}+ {\bf z}$ is contained in a fixed compact set $C\subset  \mathbb R^N$.
 Therefore we obtain from $x,{\bf z}\in {\bf 1}^\perp$, Lemma \ref{decomposition-a}, and Theorem \ref{limit-Bessel-a} that
\begin{align}
\sup_{y\in\mathbb  R^N, \> k\ge0 :\> \|y\| \le \sqrt{2k}\cdot  \| {\bf z} \|}&J_k^A(x,y+\sqrt{2k}\cdot {\bf z})= \notag\\
&=\sup_{y\in\mathbb  R^N, \> k\ge0 :\> \|y\| \le \sqrt{2k}\cdot  \| {\bf z} \|}J_k^A(x,\pi_{{\bf 1}^\perp}(y)+\sqrt{2k}\cdot {\bf z})\notag\\
&=\sup_{y\in{\bf 1}^\perp, \> k\ge0 :\> \|y\| \le \sqrt{2k}\cdot  \| {\bf z} \|}
J_k^A(x,\sqrt{2k}(\frac{y}{\sqrt{2k}} + {\bf z})) \quad <\quad \infty.
\notag\end{align}
 This estimation, (\ref{est-j-a}), and  
 (\ref{part-est-a}) readily imply that 
 the dominated convergence theorem in (\ref{density-a-limit}) works as claimed.

If we take $f$ in  Eq.~(\ref{density-a-limit}) as the constant $1$, we obtain
that the constants $\tilde c_k$ of  the probability  densities $f_k^A$
tend to
$$\tilde d(x):=\Bigl(d(x)\int_{\mathbb R^N} 
 exp\Bigl(-\frac{\|y\|^2}{2}-\frac{1}{2}\sum_{i<j}\frac{(y_i-y_j)^2}{(z_i-z_j)^2}\Bigr)\> dy\Bigr)^{-1}$$
which can be expressed explicitly in terms of $\det S_N$.
On the other hand, it follows from the proof of Theorem 2.2 in \cite{V} (see in particular Eqs.~(2.3)-(2.5) for the case $x=0$) that in our generalized case
$$\lim_{k\to\infty} \tilde c_k= e^{-\|x\|^2/2}\frac{\sqrt{N!}}{(2\pi)^{N/2}}.$$
A comparison of both limits shows that $\det S_N=N!$ as shown in Corollary 2.3 of \cite{V},
 and that the constants depending on $x$ also fit. 

If we take this convergence of the norming constants into account, we obtain from (\ref{density-a-limit})
 that 
the probability measures  $f_k^A(y)\> dy$ tend weakly to the normal distribution 
$N(0, \Sigma_N)$.
This completes the proof.
\end{proof}

We denote by $M^1(S)$ the set of probability distributions on a set $S$, and by $\mu_t$ the scaling of $\mu\in M^1(S)$ by a factor of $\sqrt{t}$, namely, $\mu_t(\{x\}):=t^{N/2}\mu(\{x\cdot\sqrt{t}\})$.

\begin{corollary}\label{clt-main-a-general}
Let $\mu\in  M^1( C_N^A)$ be an arbitrary starting distribution  on $C_N^A$.
Consider the  Bessel processes $(X_{t,k})_{t\ge0}$ of type $A_{N-1}$ on $C_N^A$ for $k\ge0$ with this starting
 distribution $\mu$.
Then
$$\frac{X_{t,k}}{\sqrt t} -  \sqrt{2k}\cdot (z_{1,N},\ldots,z_{N,N})$$
converges for $k\to\infty$ to the  $N$-dimensional distribution $\pi_{{\bf 1}}(\mu_t)*N(0, \Sigma_N)$
with the normal distribution $N(0, \Sigma_N)$,
the covariance matrix $\Sigma_N$ as in Theorem \ref{clt-main-a}, and
 the usual convolution $*$ of probability measures on $\mathbb R^N$, where  $\pi_{{\bf 1}}(\mu_t)$ is the image measure
of $\mu_t$ under the projection  $\pi_{{\bf 1}}$.
\end{corollary}

\begin{proof} If $\mu$ is a Dirac measure, say at $x\in C_N^A$, then the statement is precisely Theorem
\ref{clt-main-a}. This then leads easily to the general case; see the proof of Corollary 3.7 in \cite{V}.
\end{proof}

\section{The covariance matrices in the $A$-case}

We now study the matrices $S_N=\Sigma_N^{-1}$ from Theorems   \ref{clt-main-a} and \ref{clt-main-a-general-x}
more closely. We first determine the eigenvalues and eigenvectors.
The  eigenvectors will be described in terms of a certain finite sequence of orthogonal polynomials.
For this we introduce the empirical measures
\begin{equation}\label{empirical-measure-a}
\mu_N:=\frac{1}{N}(\delta_{z_{1,N}}+\ldots+\delta_{z_{N,N}})\in M^1(\mathbb R)
\end{equation}
of the zeros of $H_N$. We consider the associated finite sequence of orthogonal polynomials 
$\{P_n^{(N)}\}_{n=0}^{N-1}$ with positive leading coefficients and with the normalizations
\begin{equation}\label{normalization-ops-a}
\sum_{i=1}^N P_n^{(N)}(z_{i,N})^2=1 \quad\quad(n=0,\ldots, N-1).
\end{equation}
 These polynomials with $\text{deg} [P_n^{(N)}]=n$ ($n=0,\ldots,N-1$)
 are determined uniquely by Gram-Schmidt orthogonalization and normalization from the monomials $x^n$ 
 ($n=0,\ldots,N-1$) on the spaces $L^2(\mathbb R, \mu_N)$. For the  background on finite 
 sequences of orthogonal polynomials we refer to \cite{C}.
 These orthogonal polynomials satisfy a three-term recurrence relation (see \cite{C}, Section~I.4).
The normalization (\ref{normalization-ops-a}) and the orthogonality of the polynomials $P_n^{(N)}$  ensure
 that for $N\in\mathbb N$ the matrices
\begin{equation}\label{trafo-matrix-a}
T_N:= ( P_{j-1}^{(N)}(z_{i,N}))_{i,j=1,\ldots,N}
\end{equation}
are orthogonal.
 In particular,
$$  P_0^{(N)}\equiv N^{-1/2}, \quad  P_1^{(N)}(x) = \sqrt{\frac{2}{N(N-1)}}x,$$
and
$$ P_2^{(N)}(x) = c_2(x^2-(N-1)/2),\quad c_2=\frac{2}{\sqrt{N(N-1)(N-2)}}.$$
The expressions for $P_1^{(N)}(x)$ and $P_2^{(N)}(x)$ follow from orthogonality and from \eqref{hermite-zero-square-a}. 

We have the following result about the eigenvalues and eigenvectors of $S_N$:

\begin{theorem}\label{ev-a}
For each $N\ge2$, the matrix $S_N$ from Theorem \ref{clt-main-a} has the eigenvalues $1,2,\ldots,N$.
Moreover, for each $n=1,\ldots,N$, the vector
$$\bigl(P_{n-1}^{(N)}(z_{1,N}), \ldots, P_{n-1}^{(N)}(z_{N,N})\bigr)^T$$
is an eigenvector of $S_N$ for the eigenvalue $n$, i.e., $S_N= T_N \cdot \textup{diag}(1,2,\ldots,N) \cdot T_N^T.$
\end{theorem}

\begin{proof}
In the first main step of the proof we show by induction on $n=1,\ldots,N$ that $n$ is an eigenvalue
of $S_N$, and that there exists some polynomial $q_n$ of degree $n-1$ such that the vector
$$(q_n(z_{1,N}), \ldots, q_n(z_{N,N}))^T$$
is an associated eigenvector of  $S_N$. In a short second step we then will identify the polynomials $q_n$.

We start our induction with $n=1$. We observe that $(1,\ldots,1)^T$ is clearly an eigenvector
for the eigenvalue $1$. 
Moreover, if we use Lemma \ref{char-zero-A}(2), we also see easily that $(z_{1,N},\ldots,z_{N,N})^T$ is an
eigenvector for the eigenvalue $2$. 
It can be also checked with this argument and an easy computation that
$$( P_2^{(N)}(z_{1,N}) , \ldots, P_2^{(N)}(z_{N,N}))^T$$  as given above is an
eigenvector for the eigenvalue $3$. 

Let us turn to the general induction step for $n$. We use the $N\times N$-identity matrix $I_N$
 and consider the vector
$$v_n:=(z_{1,N}^{n-1} , \ldots,z_{N,N}^{n-1})^T.$$ Then the $i$-th coordinate of $(S_N-nI_N)v_n$ satisfies
\begin{align}\label{rek-a1}
((S_N&-nI_N)v_n)_i= (1-n)z_{i,N}^{n-1} + \sum_{j:\> j\ne i} \frac{z_{i,N}^{n-1}- z_{j,N}^{n-1}}{(z_{i,N}- z_{j,N})^2}
\\
&=(1-n)z_{i,N}^{n-1} + \sum_{j:\> j\ne i} \frac{z_{i,N}^{n-2}+ z_{i,N}^{n-3}z_{j,N}+\ldots+z_{j,N}^{n-2}}{z_{i,N}- z_{j,N}}
\notag\\
&=(1-n)z_{i,N}^{n-1} +(n-1)z_{i,N}^{n-2} \sum_{j:\> j\ne i}  \frac{1}{z_{i,N}- z_{j,N}}+
\notag\\
&+  \sum_{j:\> j\ne i} \frac{ z_{i,N}^{n-3}(z_{j,N}-z_{i,N}) + z_{i,N}^{n-4}(z_{j,N}^2-z_{i,N}^2)+\ldots+
1\cdot(z_{j,N}^{n-2}-z_{i,N}^{n-2})}{z_{i,N}- z_{j,N}}
\notag\\
&= -\sum_{m=1}^{n-2}\sum_{l=0}^{m-1}z_{i,N}^{n-3-l}\Big(\sum_{j=1}^Nz_{j,N}^l-z_{i,N}^l\Big),
\notag\end{align}
where the last equation follows from item (2) of Lemma \ref{char-zero-A}. If we put 
$$s_l:=\sum_{j=1}^N z_{j,N}^l \quad\quad(l=0,1,\ldots),$$
we notice that $s_l=0$ whenever $l$ is odd due to the symmetry of the zeroes of $H_N$, and we obtain
\begin{align}\label{rek-a2}
((S_N&-nI_N)v_n)_i= \sum_{m=1}^{n-2}\sum_{l=0}^{m-1}z_{i,N}^{n-3}-\sum_{l=0}^{n-3}\sum_{m=l+1}^{n-2}s_l z_{i,N}^{n-3-l}
\notag\\
&=\frac{(n-1)(n-2)}{2}z_{i,N}^{n-3}-\sum_{l=0}^{\lfloor (n-3)/2\rfloor}(n-2(l+1))s_{2l}z_{i,N}^{n-1-2(l+1)}
\notag\\
&=-\Big(N-\frac{n-1}{2}\Big)(n-2)z_{i,N}^{n-3}-\sum_{l=1}^{\lfloor (n-3)/2\rfloor}(n-2(l+1))s_{2l}z_{i,N}^{n-1-2(l+1)},
\end{align}
which is a polynomial with all terms either even or odd in $z_{i,N}$. Note that it is easy to confirm that
\begin{equation}
s_{2l}=\frac{1}{2}\sum_{m=0}^{l-1}s_{2(l-1-m)}s_{2m}-\frac{2l-1}{2}s_{2(l-1)}
\end{equation}
with $s_0=N$, meaning that the coefficients $s_l$ are functions of $N$ alone. We thus find a polynomial $r_{n-3}$ of order $n-3$ with
\begin{equation}\label{rek-a3}
(S_N-nI_N)v_n= (r_{n-3}(z_{1,N}),\ldots, r_{n-3}(z_{N,N}))^T.
\end{equation}
On the other hand, by our induction assumptions, we have polynomials $q_1,\ldots,q_{n-2}$ with $\text{deg} [q_l]=l-1$
($l=1,\ldots, n-2$) and
\begin{equation}\label{rek-a4}
(S_N-nI_N) (q_l(z_{1,N}),\ldots, q_l(z_{N,N}))^T= -(n-l)\cdot (q_l(z_{1,N}),\ldots, q_l(z_{N,N}))^T.
\end{equation}
As the $q_1,\ldots,q_{n-2}$ form a basis of the vector space $\mathbb R_{n-3}[x]$ of
 all polynomials of degree at most $n-3$, we can find a polynomial $p_{n-3}\in \mathbb R_{n-3}[x]$ that satisfies
\begin{equation}\label{rek-a5}
(S_N-nI_N) (p_{n-3}(z_{1,N}),\ldots, p_{n-3}(z_{N,N}))^T=(r_{n-3}(z_{1,N}),\ldots, r_{n-3}(z_{N,N}))^T .
\end{equation}
Therefore, the monic polynomial $q_n(x):= x^{n-1}-p_{n-3}(x)$ has the required properties. 
This completes the induction.

We finally identify the $q_n$ more explicitly. As $S_N$ is symmetric, the vectors 
$$(q_n(z_{1,N}),\ldots,q_n(z_{N,N}))^T \quad\quad(n=1,\ldots,N)$$
are orthogonal, i.e.,
$$\sum_{i=1}^N q_n(z_{i,N})\cdot q_l(z_{i,N})=0 \quad\quad(n,l=1,\ldots,N, \> n\ne l).$$
Hence, $(q_n)_{n=1,\ldots,N}$ is just a finite sequence of orthogonal
polynomials associated with the empirical measure $\mu_N$. This 
implies that the $q_n$ are equal to $P_{n-1}^{(N)}$ for $n=1,\ldots,N$ up to normalizations.
This completes the proof of the theorem.
\end{proof}

\begin{remark}\label{comparison-de}
The CLT \ref{clt-main-a} was also derived by Dumitriu and Edelman \cite{DE2} for $t=1$.
We point out that their statement contains explicit formulas for
the covariance matrix $\Sigma_N=(\sigma_{i,j}^2)_{i,j=1,\ldots,N}$ of the limit and not its inverse
$S_N=\Sigma_N^{-1}$ as in \eqref{clt-main-a}. In fact, in our notations, Theorem 3.1 of  \cite{DE2}
yields that
\begin{equation}\label{covariance-a-de}
\sigma_{i,j}^2= \frac{ \sum_{l=0}^{N-1} \tilde H_l^2(z_{i,N}) \tilde H_l^2(z_{j,N})
+ \sum_{l=0}^{N-2}\tilde H_{l+1}(z_{i,N}) \tilde H_l(z_{i,N}) \tilde H_{l+1}(z_{j,N})\tilde H_{l}(z_{j,N})}{
\sum_{l=0}^{N-1}\tilde H_l^2(z_{i,N}) \cdot \sum_{l=0}^{N-1}\tilde H_l^2(z_{j,N}) }
\end{equation}
with the \textit{orthonormal} Hermite polynomials $(\tilde H_n)_{n\ge0}$.  Theorem \ref{ev-a} and a comparison of Theorem \ref{clt-main-a} above with  Theorem 3.1 of  \cite{DE2} show that the matrix  $\Sigma_N$ as in
 (\ref{covariance-a-de}) has the form
\begin{equation}\label{covar-matrix-ortho}
 \Sigma_N= T_N \cdot \text{diag}(1,1/2,\ldots,1/N) \cdot T_N^T.
\end{equation}
Even knowing these facts, we are unable to check this statement for general dimensions $N$ 
directly via  (\ref{covariance-a-de})
even in the simplest cases like the eigenvalue $1$ with eigenvector $(1,\ldots,1)^T$.
\end{remark}

Next, we study the polynomials $P_k^{(N)}$ more closely for large dimensions $N$.
 We recapitulate the well-known fact (see e.g. \cite{G}, \cite{KM}, or \cite{D} for different proofs)
 that for $\mathbb R$-valued random variables
 $X_N$ with distributions $\mu_N$, the r.v.'s 
$\frac{1}{\sqrt{2 N}}X_N$ tend in distribution to  the r.v. $X$ which obeys the semicircle law $\mu_{sc}$, namely, the probability measure given by the density
$$\rho_{sc}(x)= \frac{2}{\pi} \sqrt{1-x^2}\cdot {\bf 1}_{[-1,1]}(x).$$
For this we recall that the odd moments of  $\mu_{sc}$  are zero while
 for $n\in\mathbb N$, the $2n$-th moments are given by
$2^{-2n}C_n$ with the Catalan numbers 
\[C_n=\frac{1}{n+1} \cdot\binom{2n}{n} \quad\quad (n\ge0);\]
see e.g.~\cite{D} or \cite{G}.
The convergence of the $\mu_N$ to  $\mu_{sc}$ above  can now be derived via the moment convergence theorem \cite{FS}.
 In fact, the following rate of convergence for the moments was given in Theorem 2 of \cite{KM}; please notice that 
\cite{KM} use a different normalization for the Hermite polynomials in their arguments.
 We have translated their results to our setting:

\begin{proposition}\label{prop-rate-moments-a}
For all $n\in\mathbb N_0$ the $n$-th moment
$$m_N(n):=E(X_N^n)= \frac{1}{N}\sum_{i=1}^N z_{i,N}^n$$
of  a random variable $X_N$ with the distribution $\mu_N$ in \eqref{empirical-measure-a}, satisfies
$$m_N(n)=\left\{\begin{array}{r@{\quad\quad}l}
(N/2)^{n/2}C_{n/2} + \frac{1}{N}\cdot f_n(N) &\text{for}\quad  n \quad \text{even}\\
0 &\text{for}\quad  n \quad \text{odd} \end{array}\right.$$
 with polynomials $f_n$ of degree at most $n/2$.
\end{proposition}

This proposition ensures that for all $n$,
\begin{equation}\label{rate-moments-a}
E\Bigl(\Bigl(\frac{1}{\sqrt{2 N}}X_N\Bigr)^n\Bigr)- E(X^n)=O(1/N) \quad\quad(N\to\infty).
\end{equation}

We now equip the vector space $\mathbb R[x]$ of all polynomials with the  positive semidefinite products
$$\langle p,q\rangle_N:= \frac{1}{N}\sum_{i=1}^N p\Big(\frac{1}{\sqrt{2 N} }z_{i,N}\Big)\cdot  q\Big(\frac{1}{\sqrt{2 N} }z_{i,N}\Big)
\quad\quad(N\in\mathbb N)$$
and
$$\langle p,q\rangle:= \frac{2}{\pi} \int_{-1}^1 \sqrt{1-x^2}\cdot p(x)q(x)\> dx= \int_{-1}^1   pq\> d\mu_{sc}$$
and study the associated orthonormal polynomials.
In the first case,  the normalization (\ref{normalization-ops-a}) shows  that these
orthonormal polynomials  $(\tilde P_n^{(N)})_{n=0,\ldots,N-1}$ satisfy
\begin{equation}\label{connection-norm-a}
 \tilde P_n^{(N)}(x) = \sqrt N \cdot  P_n^{(N)}(\sqrt{2N} \cdot x)  \quad\quad (n=0,\ldots,N-1).
\end{equation}
Moreover, by Section 4.7 of \cite{S}, 
in the second case the  orthonormal polynomials
 are the Tchebychev polynomials $(U_n)_{n\ge0}$ of the second kind with
\begin{equation}\label{def-t2}
U_n(\cos\theta)=\frac{\sin((n+1)\theta)}{\sin \theta} \quad\quad (n\in\mathbb N_0).
\end{equation}
 Proposition \ref{prop-rate-moments-a} yields:

\begin{lemma}\label{lemma-rate-polynomials-a}
For all $n\in\mathbb N_0$, and locally uniformly in $x\in\mathbb R$,
$$  \tilde P_n^{(N)}(x) - U_n(x) =O(1/N) \quad\quad(N\to\infty).$$
\end{lemma}

\begin{proof}
We first observe that $ \tilde P_0^{(N)}=1 = U_0$, $U_1(x)=2x$  and, by  (\ref{hermite-zero-square-a}),
 $ \tilde P_1^{(N)}(x)= 2 \sqrt{\frac{N}{N-1}}\cdot x$. This proves the result for $k=0,1$.

The general case follows e.g.~by induction on $n$, Proposition \ref{prop-rate-moments-a},
 and the three-term-recurrence relation of the monic orthogonal polynomials 
associated with the orthonormal polynomials $ \tilde P_n^{(N)}$ and $U_n$; see Section I.4 of \cite{C}.
In both cases, the final orthonormalizations clearly preserve the order of convergence.
\end{proof}

In the end of this section we briefly discuss some possible applications of Lemma \ref{lemma-rate-polynomials-a}
to the variances of particles of Calogero-Moser-Sutherland models,
 when we first take the limit $k\to\infty$ and then the limit $N\to\infty$.
For this we choose an index $i(N) \in\{1,\ldots,N\}$ for every $N$ and consider the variances  
 $\sigma_{i(N),i(N)}^2=\sigma^2_{i(N),i(N)}(N)$ of the $i(N)$-th particles.
Using (\ref{covar-matrix-ortho}) and (\ref{connection-norm-a}), we have
\begin{equation}\label{exact-largest-a}
\sigma_{i(N),i(N)}^2(N)= \frac{1}{N}\sum_{n=0}^{N-1} \frac{1}{n+1} \tilde P_n^{(N)}(z_{i(N),N}/\sqrt{2N}  )^2 .
\end{equation}
By Lemma \ref{lemma-rate-polynomials-a}, 
$\sigma_{i(N),i(N)}^2(N)$ should be approximately equal to
\begin{equation}\label{approx-largest-a}
\tilde\sigma_{i(N),i(N)}^2(N):= \frac{1}{N}\sum_{n=0}^{N-1} \frac{1}{n+1}  U_n(z_{i(N),N}/\sqrt{2N}  )^2.
\end{equation}

We discuss this heuristic idea for the particles in the middle of the models.
To be more precise, we consider an odd number $N=2L-1$ ($L\in\mathbb N$) of particles and investigate the 
particle with number $L$.
In this case we use the representation 
(\ref{covariance-a-de}) of \cite{DE2} and get an exact asymptotic result for $L\to\infty$.
In fact, we use $z_{L,2L-1}=0$, (\ref{covariance-a-de}), as well as 
the formulas (5.5.1) and (5.5.4) of \cite{S} on Hermite
polynomials, as well as $H_{2n+1}(0)=0$ for  $n\in\mathbb N_0$. This and  Stirling's formula imply that
$$\tilde H_{2l}(0)^2= \frac{ (2l)!}{(l!)^2 \sqrt\pi \cdot 2^{2l}} \sim \frac{1}{\pi\sqrt l}$$
and thus
\begin{align}\label{intermediate-exact}
\sigma_{L,L}^2(2L-1)&= \sum_{l=0}^{L-1} \tilde H_{2l}(0)^4 \Biggl/
 \biggl( \sum_{l=0}^{L-1} \tilde H_{2l}(0)^2\biggr)^2 \\
&\sim  \sum_{l=1}^{L-1} \frac{1}{l}\biggl/ \biggl( \sum_{l=1}^{L-1} \frac{1}{\sqrt l}\biggr)^2
\quad\sim\quad \frac{\ln L}{(2\sqrt L)^2} 
\quad=\quad \frac{\ln L}{4 L}. \notag
\end{align}
for $L\to\infty$. 
This and Theorem \ref{clt-main-a} lead to the following result:

\begin{corollary}
For $L\in\mathbb N$ let $X_{t,k}^{(L)}$ be the position of the $L$-th particle in the middle of a system with $N=2L-1$ 
particles with multiplicity $k$. Then
\begin{equation}
\frac{2\sqrt L}{\sqrt{ t\ln L}}\cdot X_{t,k}^{(L)}
\end{equation}
tends in distribution to the standard normal distribution when first the limit
 $k\to\infty$ and then the limit $L\to\infty$ is taken.
\end{corollary}

On the other hand, we now study the  approximation $\tilde\sigma_{L,L}^2(2L-1)$
 of $\sigma_{L,L}^2(2L-1)$  above. In this case we use the polynomials $U_l$ as in 
(\ref{def-t2})
and consider the fixed angle $\theta=\pi/2$  with  $z_{L,2L-1}=0=\cos\theta$. Hence,
\begin{align}\label{intermediate-exact2}
\tilde\sigma_{L,L}^2(2L-1)&=
\frac{1}{2L-1}\sum_{k=0}^{2L-2} \frac{1}{k+1}\frac{\sin^2((k+1)\pi/2)}{\sin^2 (\pi/2)}
= \frac{1}{2L-1}\sum_{k=0}^{L-1} \frac{1}{2k+1}\notag\\
&\sim \frac{\ln L}{4 L}
\end{align}
for $L\to\infty$ which fits perfectly with (\ref{intermediate-exact}). 

We finally mention that performing similar operations for the rightmost particle with number $1$ does not yield the correct asymptotics for the corresponding variance. Here $z_{1,N}$ is the largest zero of $H_N$, and
 the Theorem of Plancherel-Rotach (see e.g. (6.3.9) of \cite{S})  shows that 
\begin{equation}\label{Plancherel-Rotach}
z_{1,N}/\sqrt{2N} = 1-\frac{i_1}{6^{1/3} (2N)^{2/3}} + o(N^{-2/3})
\end{equation}
with the first positive zero $i_1$ of the Airy function $\text{Ai}(-3^{1/3} x)$, where $\text{Ai}(x)$ is the solution of the differential equation
$$\frac{d^2}{d x^2}\text{Ai}(x)-x\text{Ai}(x)=0$$
with the condition that $\text{Ai}(x)\to 0$ as $x\to\infty$. In particular, $z_{1,N}/\sqrt{2N}\in [0,1]$
for $N$ sufficiently large.
For these $N$ we now choose $\theta_N\in [0,\pi]$ with
$\cos\theta_N=z_{1,N}/\sqrt{2N}$. Then, by (\ref {Plancherel-Rotach}),
$$1-\theta_N^2/2 +O(\theta_N^4)=\cos\theta_N=z_{1,N}/\sqrt{2N}=1-\frac{i_1}{6^{1/3} (2N)^{2/3}} + o(N^{-2/3})$$
and thus
$$\theta_N=\sqrt{\frac{2^{1/3}i_1}{6^{1/3}}} \cdot N^{-1/3} +o(N^{-1/3}).$$
It can be now  shown that
\begin{align}\label{asymptotic-tilde-sigma-a}
\tilde\sigma_{1,1}^2(N)&= \frac{1}{N}\sum_{n=0}^{N-1} \frac{1}{n+1}\frac{\sin^2((n+1)\theta_N)}{\sin^2 \theta_N} \\
&\sim  \frac{1}{N \theta_N}\sum_{n=0}^{N-1} \frac{\sin^2((n+1)\theta_N)}{(n+1)\theta_N} \notag\\
&\sim  \frac{1}{2N \theta_N}\sum_{n=0}^{N-1} \frac{1}{(n+1)\theta_N} \notag\\
&\sim  \frac{\ln N }{2N\theta_N^2} \sim  \frac{6^{1/3}}{i_12^{4/3}}\cdot \frac{\ln N }{N^{1/3}}.
\notag\end{align}
As stated above, numerical experiments show that this rate does not seem to be the correct 
one for $\sigma_{1,1}^2(N)$ for $N\to\infty$. It also differs from the rate given in \cite{DE2}.

We plan to investigate the orthogonal polynomials $\tilde P_n^{(N)}(x)$ and the relations between
$\sigma_{i(N),i(N)}^2(N)$ and $\tilde\sigma_{i(N),i(N)}^2(N)$ more closely in a forthcoming paper.

\section{The $B$-case and Laguerre polynomials}

We now study the covariance matrices of the Gaussian limit of Bessel processes $(X_{t,k})_{t\ge0}$ of type $B$. The processes live in the closed Weyl chamber
$$
C_N^B:=\{x\in {\b R}^N:x_1\geq x_2\geq\cdots\geq x_N\geq 0\},
$$
and their transition semigroup generator is
\begin{equation}
L_Bf:=\frac{1}{2} \Delta f + k_1 \sum_{i=1}^{N}\frac{1}{x_i}\frac{\partial}{\partial x_i}f
+k_2 \sum_{i=1}^N\Bigl( \sum_{j\ne i} \frac{1}{x_i-x_j}\frac{1}{x_i+x_j}+\Bigr) \frac{\partial}{\partial x_i}f.
\end{equation}
As in Section~\ref{CLT-A}, the multiplicities are non-negative real parameters which we take here as $(k_1,k_2)=(\beta\cdot\nu,\beta)$ with $\nu>0$ fixed and $\beta\to\infty$; henceforth, $k$ will be regarded as an integer variable unrelated to the multiplicities. For all other related quantities, such as the transition probabilities, we refer the reader to \cite{V}. 
 In this case, the limit is related to the ordered zeroes $z_{1,N}^{(\nu-1)}\geq\cdots  \geq z_{N,N}^{(\nu-1)}$ 
of the Laguerre polynomial $L_N^{(\nu-1)}$. These polynomials are orthogonal w.r.t. the density
 $\mathrm{e}^{-x}x^{\nu-1}$ by \cite{S}. We start with the following known 
analogue of Lemma~\ref{char-zero-A} above  from \cite{S,AKM2}.

\begin{lemma}\label{char-zero-B}
For $r\in C_N^B$,  the following statements are equivalent:
\begin{enumerate}
\item[\rm{(1)}] The function 
\[W_B(y):=2\sum_{i,j: i<j} \ln(y_i^2-y_j^2) +2\nu\sum_i \ln y_i-\|y\|^2/2\]
is maximal at $r\in C_N^B$;
\item[\rm{(2)}] For $i=1,\ldots,N$, $r=(r_1,\ldots,r_N)$ satisfies
\[\frac{r_i}{2}=\sum_{j: j\ne i}\frac{2r_i}{r_i^2-r_j^2}+\frac{\nu}{r_i};\]
\item[\rm{(3)}] If $z_{1,N}^{(\nu-1)}>\ldots> z_{N,N}^{(\nu-1)}>0$ are the ordered zeroes of $L_N^{(\nu-1)}$, then
\[2(z_{1,N}^{(\nu-1)},\ldots, z_{N,N}^{(\nu-1)})=(r_{1}^2,\ldots,r_{N}^2).\]
\end{enumerate}
\end{lemma}

Using this lemma and the vector $r$ there, we have the following central limit theorem by \cite{V}.

\begin{theorem}\label{clt-main-b}
Consider the  Bessel processes $(X_{t,k})_{t\ge0}$ of type $B_{N}$ on $C_N^B$ for $k=(k_1,k_2)=(\beta\cdot\nu,\beta)$
 and $\beta,\nu>0$ with start in $x\in C_N^B$.
Then, for each $t>0$,
$$\frac{X_{t,(\beta\cdot\nu,\beta)}}{\sqrt t} -  \sqrt{\beta}\cdot r$$
converges for $\beta\to\infty$ to the centered $N$-dimensional distribution $N(0,\Sigma_N)$
with the regular covariance matrix $\Sigma_N$ with $\Sigma_N^{-1}=S_N=(s_{i,j})_{i,j=1,\ldots,N}$ given by
\begin{equation}\label{covariance-matrix-B}
s_{i,j}:=\left\{ \begin{array}{r@{\quad\quad}l}  1+\frac{2\nu}{r_i^2}+2\sum_{l\ne i} (r_i-r_l)^{-2}+2\sum_{l\ne i} (r_i+r_l)^{-2} & \text{for}\quad i=j, \\
   2(r_i+r_j)^{-2}-2(r_i-r_j)^{-2} & \text{for}\quad i\ne j.  \end{array}  \right. 
\end{equation}
The matrix $S_N$ satisfies $det\> S_N = N!2^N$.
\end{theorem}

We now
 proceed as in the previous section and determine the eigenvectors and eigenvalues of  $S_N$. 
It will be convenient for this to introduce the empirical probability measures
\begin{equation}\label{orthogonality-measure-b1}
\mu_{N,\nu}:=\frac{1}{2N(N+\nu-1)}(2z_{1,N}^{(\nu-1)}\delta_{2z_{1,N}^{(\nu-1)}}+\ldots+2z_{N,N}^{(\nu-1)}\delta_{2z_{N,N}^{(\nu-1)}}).
\end{equation}
As
\begin{equation}\label{sum-zeroes-b}
\sum_{k=1}^N z_{k,N}^{(\nu-1)}= N(N+\nu-1) 
\end{equation}
by Appendix C of \cite{AKM2}, these measures are probability measures.
Next, we study the family of orthogonal polynomials $(P_{k}^{(N,\nu)})_{k=0,\ldots,N-1}$
 with deg$[P_{k}^{(N,\nu)}]= k$
and positive leading coeffficients under the normalization
\begin{equation}\label{eq:Laguerrenormalization}
\sum_{i=1}^N 2z_{i,N}^{(\nu-1)}P_{k}^{(N,\nu)}(2z_{i,N}^{(\nu-1)})^2=1 \quad\quad(k=0,\ldots,N-1).
\end{equation}
This normalization, the notations of Lemma \ref{char-zero-B}(3), and the orthogonality of the
 $P_{k}^{(N,\nu)}$ ensure that  the matrices
\begin{equation}\label{trafo-matrix-b}
T_N:= (r_i\cdot P_k^{(N,\nu)}(r_i^2))_{i=1,\ldots,N, k=0,\ldots,N-1}
\end{equation}
are orthogonal.

The polynomials  $P_{k}^{(N,\nu)}$ can be computed explicitly for small degrees.
We have in particular,
\begin{align}
P_0^{(N,\nu)}(x)&= c_0,\quad P_1^{(N,\nu)}(x)=c_1(x-2(2N+\nu-2)),\ \textrm{and}\\
P_2^{(N,\nu)}(x)=c_2(x^2&-4(2N+\nu-3)x\notag\\
&\qquad+4[(2N+\nu-3)(2N+\nu-2)-N(N+\nu-1)])\notag
\end{align}
with the constants $c_0,c_1,$ and $c_2$ given by
\begin{align}
c_0^{-2}&=2N(N+\nu-1),\notag\\ 
c_1^{-2}&=8N(N+\nu-1)[N(N+\nu-1)-(2N+\nu-2)]\quad \text{and}\notag\\
c_2^{-2}&=32N(N+\nu-1)[N^2(N+\nu-1)^2-N(N+\nu-1)(6N+3\nu-8)\notag\\
&\qquad\qquad\qquad\qquad\qquad+2(2N+\nu-2)(2N+\nu-3)].
\end{align}
These formulae follow from direct calculations, 
and in particular the formula for $P_1^{(N,\nu)}$ stems from item (2) in Lemma~\ref{char-zero-B}.

 We characterize the matrix $S_N$ of type B in the following theorem.

\begin{theorem}\label{ev-b}
For $N\geq 2$, the matrix $S_N$ in Theorem~\ref{clt-main-b} has the eigenvalues $2, 4, \ldots, 2N$.
Moreover, for $k=0,1,\ldots,N-1$ and the eigenvalue $2(k+1)$,  an eigenvector is given by
\[(r_1P_{k}^{(N,\nu)}(r_1^2),\ldots,r_N P_{k}^{(N,\nu)}(r_N^2))^T.\]
In particular,  
$$S_N= T_N \cdot \textup{diag}(2,4,\ldots,2N) \cdot T_N^T.$$
\end{theorem}

\begin{proof}
The strategy of the proof is identical to that of Theorem~\ref{ev-a}, so we only specify the differences. 
In order to simplify the calculations that follow, we write down the action of the matrix $S_N$ on a generic vector $v$:
\begin{align}\label{rek-b0}
(S_N v)_i&=\sum_{j=1}^N s_{i,j}v_j=\Big(1+\frac{2\nu}{r_i^2}\Big)v_i+2\sum_{l:\> l\ne i}v_i\Big(\frac{1}{(r_i+r_l)^2}+\frac{1}{(r_i-r_l)^2}\Big)\\
&\qquad\qquad\qquad+2\sum_{j:\> j\ne i}v_j\Big(\frac{1}{(r_i+r_j)^2}-\frac{1}{(r_i-r_j)^2}\Big)\notag\\
&=\Big(1+\frac{2\nu}{r_i^2}\Big)v_i+4\sum_{l:\> l\ne i}v_i\frac{r_i^2+r_l^2}{(r_i^2-r_l^2)^2}-8\sum_{j:\> j\ne i}v_j\frac{r_i r_j}{(r_i^2-r_j^2)^2}\notag\\
&=\Big(1+\frac{2\nu}{r_i^2}\Big)v_i+4\sum_{l:\> l\ne i}\frac{v_i(r_i^2+r_l^2)-2v_lr_ir_l}{(r_i^2-r_l^2)^2}\notag\\
&=2v_i-4\sum_{l:\> l\ne i}\frac{v_i}{r_i^2-r_l^2}+4\sum_{l:\> l\ne i}\frac{v_i(r_i^2+r_l^2)-2v_lr_ir_l}{(r_i^2-r_l^2)^2}\notag\\
&=2\Big[v_i+2\sum_{l:\> l\ne i}\frac{v_i(r_i^2+r_l^2)-2v_lr_ir_l-v_i(r_i^2-r_l^2)}{(r_i^2-r_l^2)^2}\Big] \notag\\
&=2\Big[v_i+4\sum_{l:\> l\ne i}r_l\frac{v_ir_l-v_lr_i}{(r_i^2-r_l^2)^2}\Big].\notag
\end{align}
We used item~(2) in Lemma~\ref{char-zero-B} in the fifth line of the calculation. The induction here starts with $k=0$ and its corresponding eigenvector $(r_1,\ldots,r_N)^T$, giving 2 as the eigenvalue. For the eigenvalue 4, it can be easily verified that the corresponding eigenvector is given by
\[(r_1P_1^{(N,\nu)}(r_1^2),\ldots,r_NP_1^{(N,\nu)}(r_N^2))^T.\]

In the induction step, we consider the vector
\[v_{2k+1}:=(r_1^{2k+1},\ldots,r_N^{2k+1})^T,\]
and for $k>1$ we obtain the following using \eqref{rek-b0}:
\begin{align}\label{rek-b1}
((S_N&-2(k+1)I_N)v_{2k+1})_i= 2\Big[-kr_i^{2k+1}+4r_i\sum_{l:\> l\ne i}r_l^2\frac{r_i^{2k}-r_l^{2k}}{(r_i^2-r_l^2)^2}\Big]
\\
&= 2\Big[-kr_i^{2k+1}+4r_i\sum_{m=0}^{k-1}r_i^{2(k-1-m)}\sum_{l:\> l\ne i}\frac{r_l^{2(m+1)}}{r_i^2-r_l^2}\Big]
\notag\\
&= 2\Big[-kr_i^{2k+1}-4r_i\sum_{m=0}^{k-1}r_i^{2(k-1-m)}\sum_{l:\> l\ne i}\frac{r_i^{2(m+1)}-r_l^{2(m+1)}}{r_i^2-r_l^2}\notag\\
&\qquad+4r_i\sum_{m=0}^{k-1}r_i^{2k}\sum_{l:\> l\ne i}\frac{1}{r_i^2-r_l^2}\Big]
\notag\\
&= 2\Big[-kr_i^{2k+1}-4r_i\sum_{m=0}^{k-1}r_i^{2(k-1-m)}\sum_{l:\> l\ne i}\sum_{n=0}^m r_i^{2(m-n)}r_l^{2n}\notag\\
&\qquad+kr_i^{2k+1}(1-2\nu/r_i^2)\Big]
\notag\\
&= -4\Big[2r_i\sum_{n=0}^{k-1}\sum_{m=n}^{k-1}r_i^{2(k-1-n)}\sum_{l:\> l\ne i}r_l^{2n}+k\nu r_i^{2k-1}\Big].\notag
\end{align}
For the fourth equality we have made use of item (2) in Lemma~\ref{char-zero-B} again. Now, we introduce the sums $s_{n}=\sum_{j=1}^N r_j^n$ (note that $s_0=N$), and we write
\begin{align}\label{rek-b2}
((S_N&-2(k+1)I_N)v_{2k+1})_i= -4\Big[2r_i\sum_{n=0}^{k-1}(k-n)r_i^{2(k-1-n)}(s_{2n}-r_i^{2n})+k\nu r_i^{2k-1}\Big]\\
&= -4\Big[2\sum_{n=0}^{k-1}(k-n)s_{2n}r_i^{2(k-n)-1} - k(k+1)r_i^{2k-1}+k\nu r_i^{2k-1}\Big]\notag\\
&= -4\Big[2\sum_{n=1}^{k-1}(k-n)s_{2n}r_i^{2(k-n)-1}+k(2N+\nu-k-1)r_i^{2k-1}\Big].
\notag
\end{align}
We have used the requirement that $k>1$ for the last equality. Clearly, each term in this polynomial is of odd degree. As before, it can be confirmed directly that
\begin{equation}
s_{2l}=2\Big[\sum_{m=0}^{l-1}s_{2(l-1-m)}s_{2m}-(l-\nu)s_{2(l-1)}\Big]
\end{equation}
for $l>0$ with $s_0=N$, so all coefficients $s_{2l}$ are functions of $N$ and $\nu$. Therefore, we have a polynomial $p_{k-1}$ of degree $k-1$ such that
\begin{equation}
(S_N-2(k+1)I_N)v_{2k+1}=(r_1p_{k-1}(r_1^2),\ldots,r_N p_{k-1}(r_N^2))^T.
\end{equation}
The rest of the proof is virtually identical with that of Theorem~\ref{ev-a}, one only needs to keep track of the degrees of the polynomials in the induction step to obtain the (mutually orthogonal) eigenvectors of $S_N$. The associated polynomials are then orthogonal with respect to the measure $\mu_{N,\nu}$.
\end{proof}

Now, we study the polynomials $(P_{k}^{(N,\nu)})_k$ more closely for fixed $k$ and $\nu$ and large dimensions $N$ as in 
the preceding section.
For this we first conclude from Theorem 1 of Gawronski \cite{G} that the discrete probability measures
\begin{equation}
\frac{1}{N}(\delta_{z_{1,N}^{(\nu-1)}/4N}+\ldots+\delta_{z_{N,N}^{(\nu-1)}/4N})
\end{equation}
tend weakly to the beta distribution $\beta(1/2, 3/2)\in M^1([0,1])$, which has the density
$$f(t)=\frac{2}{\pi} t^{-1/2}(1-t)^{1/2}{\bf 1}_{[0,1]}(t).$$
As the zeroes
 $z_{i,N}^{(\nu-1)}/4N$ are contained in some compact interval for all $i,N$ (see e.g. Section 6.32 of \cite{S}),
we conclude readily from the definition of weak convergence that the measures
$$\frac{1}{4N^2}(z_{1,N}^{(\nu-1)}\delta_{z_{1,N}^{(\nu-1)}/4N}+\ldots+z_{N,N}^{(\nu-1)}\delta_{z_{N,N}^{(\nu-1)}/4N})$$
tend weakly to the measure on $[0,1]$ with density
$$\frac{2}{\pi} t^{1/2}(1-t)^{1/2}{\bf 1}_{[0,1]}(t)$$
where this measure has the mass $1/4$. Hence, after normalization, the probability measures
$$\frac{1}{N(N+\nu-1)}(z_{1,N}^{(\nu-1)}\delta_{z_{1,N}^{(\nu-1)}/4N}+\ldots+z_{N,N}^{(\nu-1)}\delta_{z_{N,N}^{(\nu-1)}/4N})$$
tend weakly to the probability measure on $[0,1]$ with density
$$\frac{8}{\pi} t^{1/2}(1-t)^{1/2}{\bf 1}_{[0,1]}(t).$$
After the transformation $[0,1]\to [-1,1]$,
$t\mapsto 2t-1$, the image of this measure is just the semicircle law $\mu_{sc}\in  M^1([-1,1])$ of the preceding section. 
 In summary we see that for random variables $Z_N$ with the distributions $\mu_{N,\nu}$ ($\nu$ fixed)
from (\ref{orthogonality-measure-b1}),  the transformed random variables 
$2\frac{Z_N}{4N}-1=\frac{Z_N}{2N}-1$
tend to  $\mu_{sc}$ in distribution. This observation in combination with the normalizations of the  $P_{k}^{(N,\nu)}$ in 
(\ref{eq:Laguerrenormalization})
proves readily and in a way similar to Lemma \ref{lemma-rate-polynomials-a}
the following convergence result for the $ P_{k}^{(N,\nu)}$ when $N\to\infty$ with  the
Tchebychev polynomials $(U_k)_{k\ge0}$  from Section 3 as limit:

\begin{lemma} For each  $\nu>0$,  and each integer $k\ge0$,
$$\lim_{N\to\infty}  \sqrt{2N(N+\nu-1)}\cdot P_{k}^{(N,\nu)}(4N(x+1))=U_k(x)$$
locally uniformly in $x$.
\end{lemma}

We expect that this limit can be used to derive additional limit results when we first take $\beta\to\infty$ 
and then $N\to\infty$, much like in the end of Section 3.

\end{document}